\documentclass[a4paper, 10pt, twoside]{article}

\usepackage{amsmath, amscd, amsfonts, amssymb, amsthm, latexsym, url, color, todonotes, bm, framed, rotating, enumerate} 
\usepackage{graphicx}
\usepackage[left=1in, right=1in, top=1.2in, bottom=1in, includefoot, headheight=13.6pt]{geometry}
\usepackage{mathtools}
\input{xy}
\xyoption{all}

\usepackage{tikz}
\usepackage{dsfont}
\usetikzlibrary{matrix}

\usepackage[colorlinks,breaklinks=true]{hyperref}
\usepackage[figure,table]{hypcap}
\hypersetup{
	bookmarksnumbered,
	pdfstartview={FitH},
	citecolor={black},
	linkcolor={black},
	urlcolor={black},
	pdfpagemode={UseOutlines}
}
\makeatletter
\newcommand\org@hypertarget{}
\let\org@hypertarget\hypertarget
\renewcommand\hypertarget[2]{%
  \Hy@raisedlink{\org@hypertarget{#1}{}}#2%
} 
\makeatother 

\newtheorem{theorem}{Theorem}[section]

\newtheorem{corollary}[theorem]{Corollary}
\newtheorem{proposition}[theorem]{Proposition}

\theoremstyle{definition}
\newtheorem{definition}[theorem]{Definition}
\newtheorem{remark}[theorem]{Remark}

\newtheorem{conjecture}[theorem]{Conjecture}

\newtheorem{notat}[theorem]{Notation}

\newcommand{\xysquare}[8]{
\[\xymatrix{
#1 \ar@{#5}[r] \ar@{#6}[d] & #2 \ar@{#7}[d]\\
#3 \ar@{#8}[r] & #4
}\]
}

\newcommand{\bb}{\mathbb}

\newcommand{\comment}[1]{}

\renewcommand{\phi}{\varphi}

\newcommand{\roi}{\mathcal{O}}

\newcommand{\xto}{\xrightarrow}

\renewcommand{\cal}{\mathcal}

\renewcommand{\ker}{\operatorname{Ker}}

\DeclareMathOperator{\codim}{codim}

\DeclareMathOperator{\Spec}{Spec}

\def\et{{\mathrm{\acute{e}t}}}

\newcommand{\CH}{C\!H}

\newcommand{\rcoeq}[3]{\xymatrix{ #1\ar@/^3mm/[r]^f \ar@/_3mm/[r]_g & #2 \ar[l]_e\ar[r] & #3}}

\usepackage[bbgreekl]{mathbbol}
\DeclareSymbolFontAlphabet{\mathbbm}{bbold}

\usepackage{fancyhdr}

\usepackage{sectsty}
\sectionfont{\Large\sc\centering}
\chapterfont{\large\sc\centering}
\chaptertitlefont{\LARGE\centering}
\partfont{\centering}

\begin{document}
\itemsep0pt

\title{On analogues of the Kato conjectures and proper base change for $1$-cycles on rationally connected varieties}

\author{Morten L\"uders }

\date{}

\maketitle

\begin{abstract} In 1986, Kato set up a framework of conjectures relating (higher) $0$-cycles and \'etale cohomology for smooth projective schemes over finite fields or rings of integers in local fields through the homology of so-called Kato complexes. In analogy, we develop a framework of conjectures for $1$-cycles on smooth projective rationally connected varieties over algebraically closed fields and for families of such varieties over henselian discrete valuation rings with algebraically closed fields. This is partly motivated by results of Colliot-Th\'el\`ene-Voisin \cite{CTVoisin2012} in dimension $3$. We prove some special cases building on recent results of Koll\'ar-Tian \cite{KollarTian}. \end{abstract}


\section{Introduction}
In \cite{Ka86} Kato set up a framework of conjectures generalising higher dimensional class field theory and the cohomological Hasse principle. Subject of these conjectures are, among others, the complexes
$$C_n^0(X):=\bigoplus_{x\in X_d}H^{d+1}(x,\bb Z/n\bb Z(d))\xto{\partial}...\xto{\partial}\bigoplus_{x\in X_1}H^2(x,\bb Z/n\bb Z(1))\xto{\partial} \bigoplus_{x\in X_0}H^1(x,\bb Z/n\bb Z(0))$$ 
for an arithmetical scheme $X$ with the direct sum $\bigoplus_{x\in X_d}H^{a+1}(x,\bb Z/n\bb Z(a)$ in degree $a$. For simplicity we focus on the case that $n$ is invertible on $X$ in which the coefficients are $\bb Z/n\bb Z(a):=\mu_n^{\otimes a}$. The homology groups of $C_n^0(X)$ are denoted by $KH^{(0)}_a(X)$. 
\begin{conjecture}\label{conjecture Kato1986}
(Kato) Let $S$ be the spectrum of a finite field or henselian discrete valuation ring with finite residue field. Let $X$ be a smooth projective scheme over $S$. Then 
$$KH^{(0)}_a(X)= \begin{cases}
      \bb Z/n\bb Z & \text{if $a=0$ and $\dim S=0$},\\
      0 & \text{otherwise}.
    \end{cases} $$
\end{conjecture}
The importance of the Conjecture \ref{conjecture Kato1986}, now proved with invertible coefficients \cite{KeS12}, comes from the fact that the groups $KH^{(0)}_a(X)$ measure the difference between the Zariski and the \'etale motivic cohomology (with finite coefficients) of $X$ in the range of higher zero-cycles, thereby immensely generalising class field theory. This has many applications to algebraic cycles, for example it implies the finiteness of Chow groups in the range of higher zero-cycles over finite fields, as predicted by Bass' finiteness conjecture. 
Motivated by results of Colliot-Th\'el\`ene-Voisin \cite{CTVoisin2012}, we propose similar conjectures for one-cycles on smooth projective rationally connected varieties over algebraically closed fields or families of such varieties over a henselian discrete valuation ring with algebraically closed residue field. Subject of these conjectures are the complexes
$$C^{-1}_n(X):=\bigoplus_{x\in X_d}H^d(x,\bb Z/n\bb Z(d-1))\xto{\partial}...\xto{\partial}\bigoplus_{x\in X_1}H^1(x,\bb Z/n\bb Z(0))\xto{\partial} \bigoplus_{x\in X_0}H^0(x,\bb Z/n\bb Z(-1)),$$
and
$$C^{-2}_n(X):=\bigoplus_{x\in X_d}H^{d-1}(x,\bb Z/n\bb Z(d-2))\xto{\partial}...\xto{\partial}\bigoplus_{x\in X_2}H^1(x,\bb Z/n\bb Z(0))\xto{\partial} \bigoplus_{x\in X_1}H^0(x,\bb Z/n\bb Z(-1)),$$
where the sum $\oplus_{x\in X_a}$ is placed in degree $a$. Let $KH^{(i)}_a(X)$ denote the homology of $C^{i}_n(X)$ in degree $a$ ($i=-2,-1$). 
\begin{conjecture}\label{conjecture_Kato}
\begin{enumerate}
\item[(1)]
Let $X$ be a smooth projective separably rationally connected variety of dimension $d$ over an algebraically closed field $k$ and $n\in k^\times$. Then
$$KH^{(-1)}_a(X)= \begin{cases}
      \bb Z/n\bb Z & \text{if $a=0$},\\
      0 & \text{otherwise}.
    \end{cases} $$
\item[(2)]
(Saito-Sato) Let $X$ be a smooth projective scheme of absolute dimension $d$ over a henselian discrete valuation ring with algebraically closed residue field $k$ and $n\in k^\times$.
Then $KH^{(-1)}_a(X)=0$ for all $a\geq 0$.
\item[(3)]
Let $X$ be a smooth projective scheme of absolute dimension $d$ over a henselian discrete valuation ring with algebraically closed residue field $k$ and $n\in k^\times$.
If the special fiber of $X$ is separably rationally connected then 
$$KH^{(-2)}_a(X)= \begin{cases}
      \bb Z/n\bb Z & \text{if $a=1$},\\
      0 & \text{otherwise}.
    \end{cases} $$
\end{enumerate}
\end{conjecture}
We summarise our main results about these conjectures in the following theorem (for an exact dedication of the conjectures and statements see Remark \ref{remark_main_thm} below).
\begin{theorem}\label{theorem_main_intro}
\begin{enumerate}
\item Let $X$ be a smooth projective separably rationally connected variety of dimension $d$ over an algebraically closed field $k$ and $n\in k^\times$.
Then $KH^{(-1)}_0(X)=\bb Z/n\bb Z$, $KH^{(-1)}_1(X)=0$ and $KH^{(-1)}_2(X)=0$ if $k=\bb C$ and the integral Hodge conjecture (IHC) holds for $1$-cycles on $X$. Furthermore, Conjecture \ref{conjecture_Kato}(i) holds if $k=\bb C$ and $d\leq 3$.
\item Conjecture \ref{conjecture_Kato}(ii) holds for $a\leq 3$. 
\item Let $X$ be a smooth projective scheme of absolute dimension $d$ over a henselian discrete valuation ring with algebraically closed residue field $k$ and $n\in k^\times$. Assume that the special fiber $X_k$ of $X$ is separably rationally connected. Then 
$$KH^{(-2)}_a(X)= \begin{cases}
      0 &  \text{if $a=0$},\\
      \bb Z/n\bb Z & \text{if $a=1$},\\
      0 &  \text{if $a=2,3,\; k=\bb C$, the IHC holds for $1$-cycles on $X_k$ } \\
      &\text{and Conjecture \ref{conjecture_Kato}(2) holds.}
    \end{cases} $$
Furthermore, if $d=4$ and $k=\bb C$ then $KH^{(-2)}_3(X)=0$ and $KH^{(-1)}_4(X)\cong KH^{(-2)}_2(X)\cong \CH^{d-1}(X,1,\bb Z/n)$.
\end{enumerate}
\end{theorem}

\begin{remark}\label{remark_main_thm} \begin{enumerate}
\item[(1)] Theorem \ref{theorem_main_intro}(i) for $d\leq 3$ is due to Colliot-Th\'el\`ene-Voisin \cite[Prop. 6.3(ii)]{CTVoisin2012} and motivates Conjecture \ref{conjecture_Kato}(i).  We expect that it is well-known to experts and therefore do not claim any originality. For example in \cite{Schreieder2021}, following Theorem 1.4, Schreieder asks if the unramified cohomology $H^{d=\dim X}_{nr}(X,\bb Z/n)=KH_d^{(-1)}(X)$ vanishes for rationally connected varieties. 
Conjecture \ref{conjecture_Kato}(2) would have the consequence that the cycle class map
$$\rho_X^{2d-2-q,d-1}:\CH^{d-1}(X,q,\bb Z/n\bb Z)\to H^{2d-2-q}_{\et}(X,\mu_n^{\otimes d-1})$$
is an isomorphism for all $q\geq 0$.
\item[(2)] Conjecture \ref{conjecture_Kato}(2) is due to Saito-Sato \cite[Conj. 2.11]{SS10} and Theorem \ref{theorem_main_intro}(ii) is proved by them in \cite[Thm. 2.13]{SS10}. 
\item[(3)]  Conjecture \ref{conjecture_Kato}(2) and (3) combined with (1) would imply that the restriction map 
$$res^{d-2,q}_{X,\bb Z/n}:\CH^{d-1}(X,q,\bb Z/n\bb Z)\to \CH^{d-2}(X_k,q,\bb Z/n\bb Z),$$
$\dim X=d$, is an isomorphism for all $q\geq 0$. In fact, our key ingredient for proving the case $a=3$ is a recent result of Koll\'ar-Tian \cite[Thm. 8]{KollarTian} which implies the surjectivity of $res^{d-2,0}_{X,\bb Z/n}$ (see Theorem \ref{theorem_Kollar_Tian} below).
\end{enumerate}
\end{remark}

\paragraph{Acknowledgement.} The author acknowledges support by the Deutsche Forschungsgemeinschaft (DFG, German Research Foundation) through the Collaborative Research Centre TRR 326 \textit{Geometry and Arithmetic of Uniformized Structures}, project number 444845124. The work was started while the author was working at Leibniz Universität Hannover.

\section{Kato complexes}
Let $X$ be an excellent scheme, $n\in \bb Z-\{0\},i\in \bb Z$ and assume that for any prime divisor $p$ of $n$ and for any $x\in X_{0}$ such that char$(k(x))=p$, we have $[k(x):k(x)^p]\geq p^i$. In \cite{Ka86}, Kato constructs complexes
\begin{multline*}
C^i_n(X): ...\to \bigoplus_{x\in X_a}H^{i+a+1}(k(x),\bb Z/n(i+a))\to ...\to 
\bigoplus_{x\in X_1}H^{i+2}(k(x),\bb Z/n(i+1))\\ \to \bigoplus_{x\in X_0}H^{i+1}(k(x),\bb Z/n(i))
\end{multline*}
in which $\bb Z/n\bb Z(a):=\mu_n^{\otimes a}$ if $n$ is invertible on $X$ and $\bb Z/n\bb Z(a):=W_r\Omega^a_{X,\log}[-a]\oplus \mu_m^{\otimes a}$ if $X$ is a smooth scheme over a field of characteristic $p>0$ and $n=mp^r$ with $p \nmid m, r\geq 0$. In particular, the complexes are always defined in $n$ is invertible on $X$.
Again the term $\oplus_{x\in X_a}H^{i+a+1}(k(x),\bb Z/n(i+a))$ is placed in degree $a$ and the homology of $C^i_n(X)$ in degree $a$ is denoted by $KH^{(i)}_a(X,\bb Z/n\bb Z)$. 
It is shown in \cite{JSS2014}, that these complexes coincide up to sign with the complexes arising from the appropriate homology theories via the niveau spectral sequence. When working in the Zariski topology, we let $\bb Z/n\bb Z(a)$ denote the motivic complex with finite coefficients as defined by Bloch's cycle.



\begin{proposition}\label{proposition_long_exact_sequence}
\begin{enumerate} Assume that we are in one of the following situations:
\item $X$ is a regular scheme over an algebraically closed field $k$ with $\dim X=d$. 
\item $A$ is a henselian discrete valuation ring with an algebraically closed residue field $k$ and $X$ is a regular scheme flat over $S=\Spec A$ with $\dim X=d$.  
\end{enumerate}
Let $\epsilon:X_{\et}\to X_{\rm Zar}$ be the natural change of sites. Then there is a distinguished triangle
$$\tau^{\geq d}R\epsilon_*\bb Z/n\bb Z(d-1)\to \bb Z/n\bb Z(d-1)\to R\epsilon_*\bb Z/n\bb Z(d-1)$$
in $D(X_{\rm Zar},\bb Z/n\bb Z)$ with $\tau^{\geq d}R\epsilon_*\bb Z/n\bb Z(d-1)\cong C_n^{-1}(X)$. In particular, there is a long exact sequence
$$...\to KH_{4}^{(-1)}(X)\to \CH^{d-1}(X,1;\bb Z/n\bb Z)\to H^{2d-3}_{\et}(X,\mu_n^{\otimes d-1})\to KH_{3}^{(-1)}(X)\to \CH^{d-1}(X)/n\to $$
$$H^{2d-2}_{\et}(X,\mu_n^{\otimes d-1})\to KH_{2}^{(-1)}(X)\to \CH^{d-1}(X,-1;\bb Z/n\bb Z)=0 \to H^{2d-1}_{\et}(X,\mu_n^{\otimes d-1})\xto{\cong} KH_1^{(-1)}(X)\to $$
$$\CH^{d-1}(X,-2;\bb Z/n\bb Z)=0 \to H^{2d}_{\et}(X,\mu_n^{\otimes d-1})\xto{\cong} KH_0^{(-1)}(X)\to 0.$$
\end{proposition}
\begin{proof}
We compare the the spectral sequence 
\begin{equation}\label{niveauspectralseq}
E^{p,q}_1=\bigoplus_{x\in X^{(p)}}CH^{r-p}(\text{Spec}k(x),-p-q;\bb Z/n\bb Z)\Rightarrow CH^r(X,-p-q;\bb Z/n\bb Z)
\end{equation} 
(see \cite[Sec. 10]{Bl86}) with the coniveau spectral sequence
\begin{equation}\label{coniveau_spseq}
E_1^{p,q}(X,\bb Z/n\bb Z(j))=\bigoplus_{x\in X^p}H^{q-p}(k(x),\bb Z/n\bb Z(j-p))\Rightarrow H^{p+q}_{\et}(X,\bb Z/n\bb Z(j))
\end{equation} 
for $r=j=d-1$. The difference is given by the complex $C_n^{-1}(X)$.
In order to see this, we use the following facts: 
\begin{itemize}
\item[(1)] $CH^{q}(\text{Spec}(k(x)),2q-p;\bb Z/n\bb Z)\cong H^p(X,\bb Z/n(q))\cong H^p_{\et}(X,\bb Z/n(q))$, where the second isomorphism holds for $p\leq q$ by the Beilinson--Lichtenbaum conjecture.
\item[(2)] $CH^{p}(\text{Spec}(k(x)),q;\bb Z/n\bb Z)$ for $p>q$ since by definition the group is generated by codimension $p$ cycles on $\Delta^q_{\Spec k(x)}$.
\item[(3)] $E_1^{\bullet,\geq d+1}=0$ follows from the following two facts about the cohomological dimension of fields:\begin{itemize}
\item  (Tate) If $F$ is a field and $K/F$ a field extension of transcendence degree $d$, then $cd(K)\leq cd(F)+d$ \cite[X, Thm. 2.1]{SGA4}.
\item Let $R$ be an integral excellent henselian local ring of dimension $d$ with residue field $F$ and fraction field $K$. Then $cd_\ell(K)\leq cd_\ell(F)+d$ for any prime number $\ell \neq ch(k)$ \cite[X, Thm. 2.3]{SGA4}, \cite[Cor. 8, p.5]{TraveauxDeGabber2014}.
\end{itemize}
\end{itemize}
\end{proof}

\begin{proposition}\label{proposition_long_exact_sequence2}
Let $A$ is a henselian discrete valuation ring with algebraically closed residue field $k$ and $X$ is a regular scheme flat over $S=\Spec A$ with $\dim X=d$. 
Let $\epsilon:X_{\et}\to X_{\rm Zar}$ be the natural change of sites. Assume that Conjecture \ref{conjecture_Kato}(2) holds.  Then there is a distinguished triangle
$$\tau^{\geq d-1}R\epsilon_*\bb Z/n\bb Z(d-2)\to \bb Z/n\bb Z(d-2)\to R\epsilon_*\bb Z/n\bb Z(d-2)$$
in $D(X_{\rm Zar},\bb Z/n\bb Z)$ with $\tau^{\geq d-1}R\epsilon_*\bb Z/n\bb Z(d-2)\cong C_n^{-2}(X)$. In particular, there is a long exact sequence
$$...\to KH_{5}^{(-2)}(X)\to \CH^{d-2}(X,1;\bb Z/n\bb Z)\to H^{2d-5}_{\et}(X,\mu_n^{\otimes d-2})\to KH_{4}^{(-2)}(X)\to \CH^{d-2}(X)/n\to $$
$$H^{2d-4}_{\et}(X,\mu_n^{\otimes d-2})\to KH_{3}^{(-2)}(X)\to \CH^{d-2}(X,-1;\bb Z/n\bb Z)=0 \to H^{2d-3}_{\et}(X,\mu_n^{\otimes d-2})\xto{\cong} KH_2^{(-2)}(X)\to $$
$$\CH^{d-2}(X,-2;\bb Z/n\bb Z)=0 \to H^{2d-2}_{\et}(X,\mu_n^{\otimes d-2})\xto{\cong} KH_1^{(-2)}(X)\to $$
$$\CH^{d-2}(X,-3;\bb Z/n\bb Z)=0 \to H^{2d-1}_{\et}(X,\mu_n^{\otimes d-2})\xto{\cong} KH_0^{(-2)}(X)\to 0.$$
\end{proposition}
\begin{proof}
This follows by the same arguments as in the proof of Proposition \ref{proposition_long_exact_sequence} from the above spectral sequences for $r=j=d-2$ and the fact that, trivialising roots of unity, $E_1^{\bullet,d}(X)\cong C_n^{(-1)}(X)$.
\end{proof}

\section{Some facts from complex geometry}
We begin by recalling some important results about the cohomology of complex algebraic varieties which are due to Colliot-Th\'el\`ene-Voisin \cite{CTVoisin2012}.
\begin{notat} If $X$ is a smooth variety over $\bb C$, then we write $X(\bb C)$ for $X$ regarded as a topological space with the complex topology and $H^p(X(\bb C),\Lambda)$ for its singular cohomology with coefficients in an abelian group $\Lambda$. If $\Lambda$ is a finite abelian group, then $H^p_{\et}(X,\Lambda)\cong H^p(X(\bb C),\Lambda)$. \end{notat}

\begin{definition}
Let $X$ be an algebraic variety over $\bb C$. We define the $i$-th unramified cohomology group with coefficients in $A$ by $H^i_{nr}(X,A):=H^0(X,\cal H^i_X(A))$ with $\cal H^i_X(A)$ the Zariski sheaf on $X$ associated to the Zariski-presheaf $U\mapsto H^i(U(\bb C),A)$.
\end{definition}

\begin{theorem}(\cite[Thm. 3.1]{CTVoisin2012})
Let $X$ be an algebraic variety over $\bb C$. For every $i\in \bb Z$ and $n>0$ there is a short exact sequence of Zariski sheaves
$$0\to \cal H^p_X(\bb Z(i))\xto{\times n}\cal H^p_X(\bb Z(i))\xto{} \cal H^p_X(\mu_n^{\otimes i})\to 0$$
on $X$. In particular, the sheaves $\cal H^p_X(\bb Z(i))$ are torsion free and so are their global sections $H^0(X,\cal H^p_X(\bb Z(i)))=H_{nr}^p(X,\bb Z(i))$.
\end{theorem}

\begin{proposition}(\cite[Prop. 3.3]{CTVoisin2012})\label{proposition_CTV_action_corresp}
Let $X$ be a smooth projective connected variety over $\bb C$ of dimension $d$. Let $A$ be an abelian group.
\begin{enumerate}
\item If the group $\CH_0(X)$ is supported on a closed subscheme $Y\subset X$ with $\dim Y=r$, then $H^0(X,\cal H^p_X(A))$ is annihilated by an integer $N\neq 0$ for $p>r$. In particular $H^0(X,\cal H^p_X(\bb Z))=0$ for $p>r$.
\item  If the group $\CH_0(X)$ is supported on a closed subscheme $Y\subset X$ with $\codim Y=r$, then $H^p(X,\cal H^d_X(A))$ is annihilated by an integer $N\neq 0$ for $p<r$.
\item If $X=\bb P^d_{\bb C}$, then $H^p(X,\cal H^q_X(A))=0$ for $p\neq q$ and $H^p(X,\cal H^p_X(A))=A$ for $p\leq d$.
\end{enumerate}
\end{proposition}

\begin{proposition}(\cite[Prop. 3.4]{CTVoisin2012})\label{proposition_birational_invariants}
For every abelian group $A$ and every $i\geq 0$ the groups $H^i(X,\cal H^d_X(A))$ are birational invariants for smooth connexted projective varieties of dimension $d$ over $\Spec(\bb C)$.
\end{proposition}

\begin{corollary}
Let $X$ be a smooth projective rational variety over $\bb C$. Then $KH^{(-1)}_a(X)=0$ for $a>0$.
\end{corollary}
\begin{proof}
We first note that the statement holds for $X=\bb P^d_{\bb C}$.
This follows immediately from Proposition \ref{proposition_CTV_action_corresp}(iii) for $A=\mu_n^{\otimes d-1}$. The corollary then follows from Proposition \ref{proposition_birational_invariants}.
\end{proof}


\begin{definition}
Let $X$ be a smooth projective connected variety over $\bb C$ of dimension $d$ and 
$$cyc^c:\CH^i(X)\to H^{2i}(X(\bb C),\bb Z(i))$$
the cycle class map. We denote the image of $cyc^c$ in $H^{2i}(X(\bb C),\bb Z(i))$ by $H^{2i}_{alg}(X(\bb C),\bb Z(i))$ and the classes of type $(i,i)$ inside of $H^{2i}(X(\bb C),\bb C(i))$ by $H^{i,i}(X)$. 
The preimage of $H^{i,i}(X)$ in $H^{2i}(X(\bb C),\bb Z(i))$ under the homomorphism 
$$H^{2i}(X(\bb C),\bb Z(i))\to H^{2i}(X(\bb C),\bb C(i)).$$
is denoted by $Hdg^{2i}(X,\bb Z)$. There is an inclusion
$H^{2i}_{alg}(X(\bb C),\bb Z(i)) \subset Hdg^{2i}(X,\bb Z)$
and we set $$Z^{2i}(X):=Hdg^{2i}(X,\bb Z)/H^{2i}_{alg}(X(\bb C),\bb Z(i)).$$ The integral Hodge conjecture (IHC) says that $Z^{2i}(X)=0$.
\end{definition}

\begin{proposition}(\cite[Prop. 3.7]{CTVoisin2012})\label{proposition_CTV3.7}
Let $X$ be a smooth projective variety over $\bb C$. Then there is a short exact sequence 
$$0\to H^3_{nr}(X,\bb Z(2))/n\to H^3_{nr}(X,\mu_n^{\otimes 2})\to Z^4(X)[n]\to 0.$$
\end{proposition}

\begin{definition}
 A variety over a field $k$ is called rationally connected (resp. separably rationally connected) if there is a family of proper algebraic curves $g:U\to Y$ whose geometric fibers are irreducible rational curves with cycle morphism $u:U\to X$ such that
$$u^{(2)}:U\times_YU\to X\times X$$
is dominant (resp. if there is a variety $Y$ and a morphism $u:Y\times \bb P^1\to X$ such that $u^{(2)}$ is dominant and smooth at the generic point).
\end{definition} 

\begin{remark} Being separably rationally connected implies being rationally connected and the converse holds if $ch(k)=0$. 
The notion of a rationally connected variety is a generalisation of rationality which instead of emphasizing global properties of $\bb P^n$ focuses on the property that there are lots of rational curves on $\bb P^n$ (\cite[p. 199]{KollarBookRatCur1995}). Rationally connected varieties have the following nice properties:
\begin{enumerate}
\item[(1)] Proper smooth separably rationally connected varieties over an algebraically closed field are algebraically simply connected (i.e. every finite \'etale cover is trivial) \cite[Sec. 3.4]{Debarre2003}. Proper smooth rationally connected varieties over the complex numbers are simply connected, i.e. $\pi_1^{\mathrm{ab}}(X)=0$ \cite[Cor. 4.18]{Debarre2001}.
\item[(2)] Rational connectedness behaves well in smooth families: let $g:X\to S$ be a smooth morphism, $S$ connected. Assume that $X_s$ is separably rationally connected for some $s\in S$. Then there is an open neighbourhood $s\in U\subset S$ such that $X_u$ is separably rationally connected for every $u\in U$. If in addition $g$ is proper and $ch(S)=0$, then $X_u$ is rationally connected for every $u\in S$ (see \cite[2.4]{KollarMiyaokaMori1992}, \cite[Ch. IV, Thm. 3.11]{KollarBookRatCur1995}).
\item[(3)] If $X$ is a smooth projective rationally connected variety  of characteristic zero, then $H^0(X,(\Omega_X^{p})^{\otimes n})=0$ for all $p,n>0$ \cite[Cor. 4.18]{Debarre2001}.
\item[(4)] If $X$ is a smooth proper separably rationally connected variety over an algebraically closed field, then for any set of distinct closed points $x_1,...,x_n\in X$ there exists a morphism $f:\bb P^1\to X$ such that $x_1,...,x_n\in f(\bb P^1)$ \cite[2.1]{KollarMiyaokaMori1992}. This implies immediately that $\CH_0(X)\cong \bb Z$. In particular, the Chow group of zero-cycles is supported on a subscheme of dimension zero which is useful for applying Proposition \ref{proposition_CTV_action_corresp}.
\end{enumerate}
\end{remark}

\begin{remark}
The following facts are known about the integral Hodge conjecture for $1$-cycles for smooth projective rationally connected varieties:
\begin{enumerate}
\item[(1)] If $X$ is a smooth complex projective threefold which is either uniruled or satisfies $K_X=\roi_X$ and $H^2(X,\roi_X)=0$, then $Z^4(X)=0$ \cite[Thm. 2]{Voisin2006}. Since rationally connected varieties are uniruled, the IHC for $1$-cycles holds in particular for rationally connected smooth projective complex threefolds.
\item[(2)] If Tate’s conjecture holds for degree $2$ Tate classes on smooth projective surfaces defined over a finite field, then the IHC holds for $1$-cycles for any smooth rationally connected variety $X$ over $\bb C$ \cite[Thm. 1.6]{Voisin2013}.
\item[(3)] If $X$ is a Fano fourfold, then $Z^6(X)=0$ \cite{Horing2011}. If $X$ is a Fano manifold of dimension $n \geq  8$ and index $n-3$, then $Z^{2n-2}(X)=0$ \cite{Floris2013}.
\end{enumerate}
\end{remark}

\section{The case of a rationally connected variety over a field}
In this section we prove some cases of Conjecture \ref{conjecture_Kato}(i) in low dimension.

\begin{proposition}\label{proposition_KH1}
Let $X$ be a smooth projective rationally connected variety of dimension $d$ over $\bb C$.
\begin{enumerate}
\item Then 
$$\bb Z/n\bb Z\cong H^{2d}_{\et}(X,\mu_n^{\otimes d-1})\xto{\cong} KH_0^{(-1)}(X)$$
and 
$$0=H^{2d-1}_{\et}(X,\mu_n^{\otimes d-1})\xto{\cong} KH_1^{(-1)}(X).$$
This holds more generally for any smooth projective separably rationally connected variety over an algebraically closed field.
\item If the integral Hodge conjecture holds for $1$-cycles on $X$, then $$ KH_2^{(-1)}(X)=0.$$
\end{enumerate} 
\end{proposition}
\begin{proof} (i) The first statement immediately follows from  Proposition \ref{proposition_long_exact_sequence}. The second  isomorphism $H^{2d-1}_{\et}(X,\mu_n^{\otimes d-1})\xto{\cong} KH_1^{(-1)}(X)$ also holds by Proposition \ref{proposition_long_exact_sequence}. In order to show that this group vanishes, we consider the short exact sequence
$$0\to \bb Z(d-1)\xto{\times n}\bb Z(d-1)\to \mu_n^{\otimes d-1}\to 0$$
and the induced exact sequence of singular cohomology groups
$$0\to H^{2d-1}_{}(X(\bb C),\bb Z(d-1))/n\to H^{2d-1}_{}(X(\bb C),\mu_n^{\otimes d-1})\to H^{2d}_{}(X(\bb C),\bb Z(d-1))[n]\to 0.$$
By the comparison theorem for \'etale and singular cohomology, we have that $H^{2d-1}_{}(X(\bb C),\mu_n^{\otimes d-1})\cong H^{2d-1}_{\et}(X,\mu_n^{\otimes d-1})$. Since $H^{2d}_{}(X(\bb C),\bb Z(d-1))\cong \bb Z$ we get that $H^{2d}_{}(X(\bb C),\bb Z(d-1))[n]=0$. Furthermore, $H^{2d-1}_{}(X(\bb C),\bb Z(d-1))=0$.
Indeed, by Poincar\'e duality
$$H^{2d-1}_{}(X(\bb C),\bb Z(d-1))\cong H_1(X(\bb C),\bb Z)\cong \pi_1^{\mathrm{ab}}(X(\bb C)).$$
But proper smooth rationally connected varieties over the complex numbers are simply connected, i.e. $\pi_1^{\mathrm{ab}}(X)=0$ \cite[Cor. 4.18]{Debarre2001}. This implies that $H^{2d-1}_{\et}(X,\mu_n^{\otimes d-1})=0$. More generally, over an arbitrary algebraically closed field $k$, we have that 
$$H^{2d-1}_{\et}(X,\mu_n^{\otimes d-1})=H^1(X,\mu_n)^\vee\cong H^1(X,\bb Z/n\bb Z)^\vee=\pi_1^{\mathrm{ab}}(X)/n=0$$
since proper smooth separably rationally connected varieties over an algebraically closed field are algebraically simply connected (i.e. every finite \'etale cover is trivial) \cite[Sec. 3.4]{Debarre2003}.

(ii) By Proposition \ref{proposition_long_exact_sequence} we have an exact sequence
$$\CH^{d-1}(X)/n\to H^{2d-2}_{\et}(X,\mu_n^{\otimes d-1})\to KH_{2}(X)\to 0.$$
We show that the integral Hodge conjecture for $1$-cycles implies that the first map is surjective. We first show that the group $H^{2d-2}_{\et}(X,\mu_n^{\otimes d-1})\cong H^{2d-2}(X(\bb C),\mu_n^{\otimes d-1})$ is isomorphic to $H^{2d-2}(X(\bb C),\bb Z(d-1))/n$. Indeed, the short exact sequence
$$0\to \bb Z(d-1)\xto{\times n}\bb Z(d-1)\to \mu_n^{\otimes d-1}\to 0$$
implies that the sequence
$$0\to H^{2d-2}_{}(X(\bb C),\bb Z(d-1))/n\to H^{2d-2}_{}(X(\bb C),\mu_n^{\otimes d-1})\to H^{2d-1}_{}(X(\bb C),\bb Z(d-1))[n]\to 0$$
is exact. But, as we saw above, $H^{2d-1}_{}(X(\bb C),\bb Z(d-1))=0$. This implies that it suffices to show that the map $\CH^{d-1}(X)\to H^{2d-2}_{}(X(\bb C),\bb Z(d-1))$ is surjective. Since $X$ is rationally connected, $H^{2d-2}_{}(X(\bb C),\bb C(d-1))\cong H^{d-1,d-1}(X)$ and therefore $H^{2d-2}_{}(X(\bb C),\bb Z(d-1))=Hdg^{2(d-1)}$. But by assumption $$Hdg^{2(d-1)}/\rm{im}(cyc^{2d-2,d-1})=Z^{2d-2}(X)=0.$$
\end{proof}

\begin{remark}
In the course of the proof of (ii) we saw that for a smooth projective rationally connected variety over $\bb C$ 
$$\CH^{d-1}(X)\to H^{2d-2}_{}(X(\bb C),\bb Z(d-1))\; \text{is surjective} \Leftrightarrow Z^{2d-2}(X)=0.$$
In particular, the surjectivity of the map $\CH^{d-1}(X)/n\to H^{2d-2}_{\et}(X,\mu_n^{\otimes d-1})$ is strictly weaker than the integral Hodge conjecture for $1$-cycles.
\end{remark}


The following proposition is due to Colliot-Th\'el\`ene-Voisin \cite[Prop. 6.3]{CTVoisin2012}. We spell out the details of the proof in our framework.
\begin{proposition}\label{proposition_low_dimension}\label{proposition_KH2}
Let $X$ be a smooth projective rationally connected variety of dimension $d$ over $\bb C$.
\begin{enumerate}
\item If $d=2$, then $KH^{(-1)}_a(X)=0$ for $a>0$.
\item If $d=3$, then $KH^{(-1)}_a(X)=0$ for $a>0$.
\end{enumerate} 
\end{proposition}
\begin{proof}
(i) Let $d=2$. By Proposition \ref{proposition_KH1} and \ref{proposition_KH2} it suffices to show that the integral Hodge conjecture holds for one-cycles on $X$, i.e. $Z^2(X)=0$. This follows from the Lefschetz theorem on $(1,1)$-classes. 

(ii) Let $\dim X=3$.
By Proposition \ref{proposition_KH1} and \ref{proposition_KH2} it suffices to show that
$KH_3^{(-1)}(X)=H^3_{nr}(X,\mu_n^{\otimes 2})=0.$
By Proposition \ref{proposition_CTV3.7} there is a short exact sequence
$$0\to H^3_{nr}(X,\bb Z(2))/n\to H^3_{nr}(X,\mu_n^{\otimes 2})\to Z^4(X)[n]\to 0.$$
The group $H^3_{nr}(X,\bb Z(2))/n=0$ by Proposition \ref{proposition_CTV_action_corresp}(i) since the Chow group of zero-cycles of $X$, being rationally connected, is supported in dimension $0$. The group $Z^4(X)[n]=0$ since the integral Hodge conjecture holds for uniruled threefolds, and therefore rationally connected threefolds, by \cite{Voisin2006}. This implies that $H^3_{nr}(X,\mu_n^{\otimes 2})=0$.

\end{proof}



\section{The case of a family of rationally connected varieties over a dvr}

Let $Y$ be a $k$-variety. We denote by $A_m(Y)$ the group of $m$-cycles modulo algebraic equivalence. Algebraic equivalence is coarser than rational equivalence and we define $R_q(Y)$ by the exact sequence
$$0\to R_q(Y)\to \CH_q(Y)\to A_q(Y)\to 0.$$
Let $J_C$ denote the Jacobian of a smooth curve over $k$. If $J_C(k)/n=0$ for any such $C$, then $R_q(Y)/n=0$ and $\CH_q(Y)/n\cong A_q(Y)/n$. This holds in particular if $k$ is algebraically closed \cite[Lem. 7.10]{BO74}.

\begin{theorem}\label{theorem_Kollar_Tian}
Let $A$ be a henselian discrete valuation ring with algebraically closed residue field $k$. Let $ X$ be a smooth projective scheme of relative dimension $d$ over $S$ with separably rationally connected special fiber $X_k$ and generic fiber $X_K$. Then there is a commutative diagram of surjections
$$\xymatrix{
 \ar@{->>}[d]^-{}_{}   \CH_{2}(X)/n  \ar@{->>}[r]^{res}_-{} & \CH_1(X_k)/n.  \\
\CH_1(X_K)/n  \ar@{->>}[ru]^-{}_{sp}   &     \
}$$
\end{theorem}
\begin{proof}
The statement follows from the commutative diagram
$$\xymatrix{
\CH_1(X_K)/n \ar@{->>}[d]^-{}_{}  & \CH_{2}(X)/n \ar@{->>}[l]^-{}_{} \ar[r]^{}_-{} & \CH_1(X_k)/n \ar[d]^-{\cong} \\
A_1(X_K)/n  \ar[rr]_-{\cong}  &   & A_1(X_k)/n .  \
}$$
in which the right vertical map is an isomorphism since $k$ is algebraically closed and the lower horizontal map is an isomorphism by \cite[Thm. 8]{KollarTian}.
\end{proof}

\begin{theorem}
Let $X$ be a smooth projective scheme of absolute dimension $d$ over a henselian discrete valuation ring with algebraically closed residue field $k$ and $n\in k^\times$.
Then the following holds:
\begin{enumerate}
\item $KH^{(-1)}_a(X)=0$ for $a\leq 3$. 
\item Assume that the special fiber $X_k$ of $X$ is separably rationally connected. Then 
$$KH^{(-2)}_a(X)= \begin{cases}
      0 &  \text{if $a=0$},\\
      \bb Z/n\bb Z & \text{if $a=1$},\\
      0 &  \text{if $a=2,3,\; k=\bb C$, the IHC holds for $1$-cycles on $X_k$ } \\
      &\text{and Conjecture \ref{conjecture_Kato}(2) holds.}
    \end{cases} $$
\item Assume that the special fiber $X_k$ of $X$ is separably rationally connected and $k=\bb C$. If $d=4$, then $KH^{(-2)}_3(X)=0$ and $KH^{(-1)}_4(X)\cong KH^{(-2)}_2(X)\cong \CH^{d-1}(X,1,\bb Z/n)$.
\end{enumerate} 
\end{theorem}
\begin{proof}
(i) This is \cite[Thm. 2.13]{SS10}. (ii) This is trivial for $a=0$. By (i), trivialising roots of unity, and proper base change and Proposition \ref{proposition_KH1} we have that $$KH^{(-2)}_1(X)\cong H_{\et}^{2d-2}(X,\mu_n^{d-2})\cong H_{\et}^{2d-2}(X_k,\mu_n^{d-2})\cong \bb Z/n\bb Z$$ 
and
$$KH^{(-2)}_2(X)/d_2^{d-4,4}(KH^{(-1)}_4(X))\cong H_{\et}^{2d-3}(X,\mu_n^{d-2})\cong H_{\et}^{2d-3}(X_k,\mu_n^{d-2})\cong 0$$ 
(see also the coniveau spectral sequence below). In particular, if Conjecture \ref{conjecture_Kato}(2) holds for $a=4$, then $KH^{(-2)}_2(X)=0$.

For the case $a=3$ we first note that Theorem \ref{theorem_Kollar_Tian} and the IHC hold for $1$-cycles on $X_k$ imply that $\rho^{2d-4,d-2}_X:\CH_2(X)/n\to  H^{2d-4}_{\et}(X,\mu_n^{\otimes d-2})$ is surjective. This follows from the commutative diagram
$$\xymatrix{
\CH_2(X)/n \ar[d]^-{}_{}    \ar@{->>}[r]^{res}_-{} & \CH_1(X_k)/n \ar@{->>}[d]^-{} \\
 H^{2d-4}_{\et}(X,\mu_n^{\otimes d-2}) \ar[r]^{}_-{\cong} & H^{2d-4}_{\et}(X_k,\mu_n^{\otimes d-2})  \
}$$
in which the lower horizontal map is an isomorphism by proper base change, the right vertical map is surjective by Proposition \ref{proposition_KH1}(ii) and $res$ is surjectuve by Theorem \ref{theorem_Kollar_Tian}.
If we assume that Conjecture \ref{conjecture_Kato}(2) holds then we get the long exact sequence (Prop. \ref{proposition_long_exact_sequence2})
$$KH_{4}^{(-2)}(X)\to \CH^{d-2}(X)/n\xto{\rho^{2d-4,d-2}_X} H^{2d-4}_{\et}(X,\mu_n^{\otimes d-2})\to KH_{3}^{(-2)}(X)\to \CH^{d-2}(X,-1;\bb Z/n\bb Z)=0$$
and therefore that $KH_{3}^{(-2)}(X)=0$. 

(iii) Let us assume that $d=4$. Consider the spectral sequence 
$$E_1^{p,q}(X,\bb Z/n\bb Z(j))=\bigoplus_{x\in X^p}H^{q-p}(k(x),\bb Z/n\bb Z(j-p))\Rightarrow H^{p+q}_{\et}(X,\bb Z/n\bb Z(j))$$
for $j=d-2$. There is a filtration
$$0\subset F_2\subset F_1\subset H^{2d-4}_{\et}(X,\bb Z/n\bb Z(d-2))$$
with quotients
$$F_2=(\CH_2(X)/n)/d_2^{0,3}(KH_{4}^{(-2)}(X)),$$
$$F_1/F_2=E_\infty^{d-3,d-1}\cong E_2^{d-3,d-1}\cong KH_{3}^{(-2)}(X),$$
$$H^{2d-4}_{\et}(X,\bb Z/n\bb Z(d-2))/F_1=E_\infty^{d-4,d}\cong E_2^{d-4,d}\cong \ker(d_2^{0,4}).$$
Note that  the isomorphism $E_\infty^{d-3,d-1}\cong E_2^{d-3,d-1}$ holds since $d\leq 4$. The surjectivity of the inclusion $F_2\to H^{2d-4}_{\et}(X,\bb Z/n\bb Z(d-2))$ implies that $F_2=F_1=H^{2d-4}_{\et}(X,\bb Z/n\bb Z(d-2))$ and therefore $KH_{3}^{(-2)}(X)=0$ and $\ker(d_2^{0,4})=0$ which implies that $KH_{4}^{(-1)}(X)\cong KH^{(-2)}_2(X)\cong \CH^{d-1}(X,1,\bb Z/n)$ since $H_{\et}^{2d-3}(X,\mu_n^{d-2})\cong H_{\et}^{2d-3}(X_k,\mu_n^{d-2})\cong 0$.
\end{proof}

\begin{conjecture} Let the notation be as in Theorem \ref{theorem_Kollar_Tian}. Then the natural homomorphisms $\CH_2(X)/n\to \CH_1(X_K)/n\to A_1(X_K)/n$ are isomorphisms.
\end{conjecture}

\begin{remark} Conjecture \ref{conjecture_Kato} is in the case of the theorem closedly related to \cite[Conj. 10.3]{KerzEsnaultWittenberg2016} which asserts in particular that $\CH^{d-1}(X,1,\bb Z/n)\cong \CH^{d-1}(X_k,1,\bb Z/n)$ and the latter group is zero if $d=4$ if $X_k$ is rationally connected (see also \cite{Lu16}). 

It may of course be expected that Conjecture \ref{conjecture_Kato} does not just hold under the assumption that $X$ (or the special fiber of $X$) is rationally connected but also in some other special cases. Indeed, for example \cite[Prop. 6.3]{CTVoisin2012} holds for any smooth projective threefold $X_{\bb C}$ which is uniruled without torsion in the Picard group and such that $H^2(X,\roi_X)=0$. We have tried to phrase the conjecture in the most simple and uniform way. One approach to constructing more examples, in which the conjectures hold, is the following theorem and to use the main theorem of \cite[Sec. 3]{SS10}. \end{remark}
\begin{theorem}
Let $X$ be a smooth projective variety of dimension $d\geq 3$ over an algebraically closed field $k$.
Assume that there is a chain of closed subvarieties
$$X_2\subset X_3\subset \dots \subset X_{d-1}\subset X_d=X$$
such that $\dim X_i=i$ and such that each pair $(X_i,X_{i-1})$ is ample, meaning that $X_i\setminus X_{i-1}=U_i$ is an affine open subset of $X_i$. Assume furthermore that $X_2$ is smooth and rationally connected.
Then the cycle class map $$\rho_{X}^{2d-2,d-1}:\CH_1(X)/n\to H^{2d-2}_{\et}(X,\mu_n^{d-1})$$ is surjective.
\end{theorem}
\begin{proof}
Since $X_2$ is rationally connected of dimension $2$, by \ref{proposition_low_dimension} there is an isomorphism $\CH^1(Y_k)/n\to H^2_{\et}(Y_k,\mu_n)$. If $d=3$, then there is a commutative diagram with exact lower row
$$\xymatrix{
\CH_1(X_2)/n \ar[d]^-{\cong}_{\rho_{X_2}^{2,1}} \ar[r]^{}_-{} & \CH_{1}(X_3)/n \ar[d]^-{}_{\rho_{X_3}^{4,2}} &  \\
H^{2}(X_2,\mu_n^{}) \ar[r]_-{}  & H^{4}(X_3,\mu_n^{\otimes 2}) \ar[r]_-{}  & H^{4}(U_3,\mu_n^{\otimes 2}) .  \
}$$
Since the cohomological dimension of the affine scheme $U_3$ is three, $H^{4}(U_3,\mu_n^{\otimes 2})=0$. Therefore the surjectivity of  $\rho_{X_2}^{2,1}$ implies the surjectivity of  $\rho_{X_3}^{4,2}$. For $d\geq i\geq 4$, we consider the diagram
$$\xymatrix{
\CH_1(X_{i-1})/n \ar@{->>}[d]^-{}_{\rho_{X_{i-1}}^{2i-2,i-1}} \ar[r]^{}_-{} & \CH_{1}(X_i)/n \ar[d]^-{}_{\rho_{X_i}^{2i,i}} &  \\
H^{2i-2}(X_{i-1},\mu_n^{i-1}) \ar[r]_-{}  & H^{2i}(X_i,\mu_n^{\otimes i}) \ar[r]_-{}  & H^{2i}(U_{i},\mu_n^{\otimes 2}) .  \
}$$
in which by cohomological dimension $H^{2i}(U_{i},\mu_n^{\otimes 2})=0$ and therefore the surjectivity of  $\rho_{X_{i-1}}^{2i-2,i-1}$ implies the surjectivity of  $\rho_{X_i}^{2i,i}$. This inductively implies the statement of the theorem.
\end{proof}

\bibliographystyle{acm}
\bibliography{Bibliografie}

\end{document}